\documentclass{amsart}
\usepackage{tikz}
\usepackage{hyperref}   
\usepackage{amssymb,xcolor}
\usepackage{stmaryrd} % \owedge

\usepackage{mathtools}  
\mathtoolsset{showonlyrefs}
\theoremstyle{plain}

\theoremstyle{plain}
\newtheorem{theorem}{Theorem}[section]
\theoremstyle{plain}
\newtheorem{lemma}[theorem]{Lemma}
\theoremstyle{remark}

\theoremstyle{plain}
\newtheorem{proposition}[theorem]{Proposition}
\numberwithin{equation}{section}

\begin{document}

\title{A perturbation of spacetime Laplacian equation}

\author{Xiaoxiang Chai}
\address{Korea Institute for Advanced Study, Seoul 02455, South Korea}
\email{xxchai@kias.re.kr}

\begin{abstract}
  We study a perturbation
  \begin{equation}
    \Delta u + P | \nabla u| = h | \nabla u|,
  \end{equation}
  of spacetime Laplacian equation in an initial data set $(M, g, p)$ where $P$
  is the trace of the symmetric 2-tensor $p$ and $h$ is a smooth function.
\end{abstract}

{\maketitle}

Stern {\cite{stern-scalar-2019}} introduced a level set method of harmonic
maps and he used it to simplify proofs of some important results in scalar
curvature geometry of three manifolds. See {\cite{bray-scalar-2019}} for a
Neumann boundary version. Later developments include simplified proofs of
positive mass theorem {\cite{bray-harmonic-2019,hirsch-spacetime-2020}},
Gromov's dihedral rigidity conjecture for three dimension cubes and mapping
torus of hyperbolic 3-manifolds {\cite{chai-scalar-2021}}, and hyperbolic
positive mass theorem {\cite{bray-spacetime-2021}}.

The first approach to positive mass theorems used stable minimal surface
{\cite{schoen-proof-1979}}, Jang equation {\cite{schoen-proof-1981}} and
stable marginally outer trapped surfaces {\cite{eichmair-spacetime-2016}}. The
hyperbolic positive mass theorem could be established via a study of constant
mean curvature 2 surface in asymptotically hyperbolic manifolds under extra
assumption on the mass aspect function (see {\cite{andersson-rigidity-2008}}).
See also a Jang equation proof {\cite{sakovich-jang-2021}}.

Gromov {\cite{gromov-metric-2018,gromov-four-2021}} further generalized
minimal surface approach to stable $\mu$-bubbles. The merit of a stable
$\mu$-bubble is possibility of weaker scalar curvature condition. We take a
simplest example: if a three manifold $(M, g)$ admits a function $h \in
C^{\infty} (M)$ such that the scalar curvature $R_g$ satisfies
\begin{equation}
  R_g + h^2 - 2 | \nabla^g h| \geqslant 0, \label{weak scalar non-negativity}
\end{equation}
then the topology of a stable $\mu$-bubble of prescribed mean curvature $h$ in
$(M^3, g)$ can be still classified using the stability condition and
{\cite{fischer-colbrie-structure-1980}}.

We generalize the (spacetime) harmonic function approach of
{\cite{bray-harmonic-2019,hirsch-spacetime-2020}} to deal the weaker condition
similar to \eqref{weak scalar non-negativity}. An initial data set $(M, g, p)$
is a Riemannian manifold $(M, g)$ equipped with an extra symmetric 2-form $p$.
We write $P =\ensuremath{\operatorname{tr}}p$, the spacetime Hessian
{\cite{hirsch-spacetime-2020}} is defined to be
\begin{equation}
  \bar{\nabla}^2 u = \nabla^2 u + p | \nabla u|, u \in C^2 (M) .
  \label{spacetime hessian}
\end{equation}
Given a smooth function $h$ on $M$, we study the solution to the equation
\begin{equation}
  \Delta u + P | \nabla u| = h | \nabla u| . \label{h spacetime harmonic}
\end{equation}
The equation is a perturbation to the spacetime harmonic function, and the
solution $u$ is spacetime harmonic if $h = 0$. We follow {\cite[Proposition
3.2]{hirsch-spacetime-2020}} in detail and establish the analogous
proposition.

\begin{proposition}
  \label{level set method}Let $(U, g, p)$ be a 3-dimensional oriented compact
  initial data set with smooth boundary $\partial U$, having outward unit
  normal $\eta$. Let $u : U \to \mathbb{R}$ be a spacetime harmonic function,
  and denote the open subset of $\partial U$ on which $| \nabla u| \neq 0$ by
  $\partial_{\neq 0} U$. If $\bar{u}$ and $\underline{u}$ denote the maximum
  and minimum values of $u$ and $\Sigma_t$ are $t$-level sets, then
\begin{align}
& \int_{\partial_{\neq 0} U} (\partial_{\eta} | \nabla u| + p (\nabla u,
\eta)) \mathrm{d} A \\
\geqslant & \int_{\underline{u}}^{\bar{u}} \int_{\Sigma_t} \left(
\frac{1}{2}  \frac{| \bar{\nabla}^2 u|^2}{| \nabla u|^2} + (\mu + J (\nu)
+ h^2 - 2 h P + 2 \langle \nu, \nabla h \rangle) - K \right) \mathrm{d} A
\mathrm{d} t, \label{level set method main formula}
\end{align}
  where $\nu = \tfrac{\nabla u}{| \nabla u|}$, $K$ is Gauss curvature of the
  level set and $\mathrm{d} A$ is the area element.
\end{proposition}

\begin{proof}
  Recall the Bochner formula
  \begin{equation}
    \frac{1}{2} \Delta | \nabla u|^2 = | \nabla^2 u|^2
    +\ensuremath{\operatorname{Ric}} (\nabla u, \nabla u) + \langle \nabla u,
    \nabla \Delta u \rangle .
  \end{equation}
  For $\delta > 0$ set $\phi_{\delta} = (| \nabla u|^2 +
  \delta)^{\tfrac{1}{2}}$, and use the Bochner formula to find
\begin{align}
\Delta \phi_{\delta} = & \frac{\Delta | \nabla u|^2}{2 \phi_{\delta}} -
\frac{| \nabla | \nabla u |^2 |^2}{4 \phi_{\delta}^3} \\
\geqslant & \frac{1}{\phi_{\delta}}  (| \nabla^2 u|^2 - | \nabla | \nabla
u||^2 +\ensuremath{\operatorname{Ric}}(\nabla u, \nabla u) + \langle
\nabla u, \nabla \Delta u \rangle) . \label{laplacian regularized
gradient}
\end{align}
  On a regular level set $\Sigma$, the unit normal is $\nu = \tfrac{\nabla
  u}{| \nabla u|}$ and the second fundamental form is given by $A_{i j} =
  \frac{\nabla_i \nabla_j u}{| \nabla u|}$, where $\partial_i$ and
  $\partial_j$ are tangent to $\Sigma$. We then have
  \begin{equation}
    |A|^2 = | \nabla u|^{- 2}  (| \nabla^2 u|^2 - 2| \nabla | \nabla u||^2 +
    [\nabla^2 u (\nu, \nu)]^2),
  \end{equation}
  and the mean curvature satisfies
  \begin{equation}
    | \nabla u| H = \Delta u - \nabla^2_{\nu} u \label{laplacian decomposition
    on level set}
  \end{equation}
  Furthermore by taking two traces of the Gauss equations
  \begin{equation}
    2 \mathrm{Ric} (\nu, \nu) = R_g - R_{\Sigma} - |A|^2 + H^2,
  \end{equation}
  where $R_g$ is the scalar curvature of $U$ and $R_{\Sigma}$ is the scalar
  curvature of the level set $\Sigma$. Combining these formulas with
  \eqref{laplacian regularized gradient} produces
\begin{align}
\Delta \phi_{\delta} \geqslant & \frac{1}{\phi_{\delta}}  \left( |
\nabla^2 u|^2 - | \nabla | \nabla u||^2 + \langle \nabla u, \nabla \Delta
u \rangle + \frac{| \nabla u|^2}{2}  (R_g - R_{\Sigma} + H^2 - |A|^2)
\right) \\
= & \frac{1}{2 \phi_{\delta}}  (| \nabla^2 u|^2 + (R_g - R_{\Sigma}) |
\nabla u|^2 + 2 \langle \nabla u, \nabla \Delta u \rangle + (\Delta u)^2 -
2 (\Delta u) \nabla_{\nu}^2 u) .
\end{align}
  Let us now replace the Hessian with the spacetime Hessian via the relation
  \eqref{spacetime hessian}, and utilize \eqref{h spacetime harmonic} to find
\begin{align}
& \Delta \phi_{\delta} \\
\geqslant & \frac{1}{2 \phi_{\delta}}  (| \bar{\nabla}^2 u - p| \nabla
u||^2 + (R_g - R_{\Sigma}) | \nabla u|^2 + 2 \langle \nabla u, \nabla (-
P| \nabla u| + h| \nabla u|) \rangle  \\
& \quad + (- P| \nabla u| + h| \nabla u|)^2 - 2 (- P| \nabla u| + h|
\nabla u|) \nabla_{\nu}^2 u) .
\end{align}
  Noting that
  \begin{equation}
    \langle \nabla u, \nabla | \nabla u| \rangle = \frac{1}{2}  \langle \nu,
    \nabla | \nabla u|^2 \rangle = u^i \nabla_{i \nu} u = | \nabla u|
    \nabla_{\nu}^2 u,
  \end{equation}
  and expanding $| \bar{\nabla}^2 u - p| \nabla u||^2$, we have that
\begin{align}
& \Delta \phi_{\delta} \\
\geqslant & \frac{1}{2 \phi_{\delta}}  (| \bar{\nabla}^2 u|^2 - 2 \langle
p, \nabla^2 u \rangle | \nabla u| - |p|_g^2 | \nabla u|^2 + (R_g -
R_{\Sigma}) | \nabla u|^2  \\
& \quad + 2 | \nabla u|  \langle \nabla u, - \nabla P + \nabla h \rangle
+ (- P| \nabla u| + h| \nabla u|)^2) .
\end{align}
  Regrouping using the energy density $2 \mu = R_g + P^2 - |p|_g^2$,
\begin{align}
& \Delta \phi_{\delta} \\
\geqslant & \frac{1}{2 \phi_{\delta}} [| \bar{\nabla}^2 u |^2 - 2 \langle
p, \nabla^2 u \rangle | \nabla u | + (2 \mu - R_{\Sigma}) | \nabla u|^2 -
2 | \nabla u|  \langle \nabla u, \nabla P \rangle \\
& \quad + | \nabla u|^2 (h^2 - 2 h P + 2 \langle \nu, \nabla h \rangle] .
\label{final non-integrated}
\end{align}
  Consider an open set $\mathcal{A} \subset [\underline{u}, \bar{u}]$
  containing the critical values of $u$, and let $\mathcal{B} \subset
  [\underline{u}, \bar{u}]$ denote the complementary closed set. Then
  integrate by parts to obtain
  \begin{equation}
    \int_{\partial U} \langle \nabla \phi_{\delta}, v \rangle = \int_U \Delta
    \phi_{\delta} = \int_{u^{- 1} (\mathcal{A})} \Delta \phi_{\delta} +
    \int_{u^{- 1} (\mathcal{B})} \Delta \phi_{\delta} .
  \end{equation}
  According to Kato type inequality {\cite[Lemma 3.1]{hirsch-spacetime-2020}}
  and \eqref{laplacian regularized gradient} there is a positive constant
  $C_0$, depending only on $\ensuremath{\operatorname{Ric}}_g$, $P - h$ and
  its first derivatives such that
  \begin{equation}
    \Delta \phi_{\delta} \geqslant - C_0  | \nabla u| .
  \end{equation}
  An application of the coarea formula to $u : u^{- 1} (\mathcal{A}) \to
  \mathcal{A}$ then produces
  \begin{equation}
    - \int_{u^{- 1} (\mathcal{A})} \Delta \phi_{\varepsilon} \leqslant C_0 
    \int_{u^{- 1} (\mathcal{A})} | \nabla u| = C_0  \int_{t \in \mathcal{A}} |
    \Sigma_t | \mathrm{d} t, \label{error estimate}
  \end{equation}
  where $| \Sigma_t |$ is the 2-dimensional Hausdorff measure of the $t$-level
  set $\Sigma_t$. Next, apply the coarea formula to $u : u^{- 1} (\mathcal{B})
  \to \mathcal{B}$ together with \eqref{final non-integrated} to obtain
\begin{align}
& \int_{u^{- 1} (\mathcal{B})} \Delta \phi_{\delta} \\
\geqslant & \frac{1}{2}  \int_{t \in \mathcal{B}} \int_{\Sigma_t} \frac{|
\nabla u|}{\phi_{\delta}}  \left[ \frac{| \bar{\nabla}^2 u|^2}{| \nabla
u|^2} + 2 \mu - R_{\Sigma_t} - \frac{2}{| \nabla u|} (\langle p, \nabla^2
u \rangle + \langle \nabla u, \nabla P \rangle) \right] \mathrm{d} A
\mathrm{d} t \\
& \quad + \frac{1}{2}  \int_{t \in \mathcal{B}} \int_{\Sigma_t}
\frac{1}{\phi_{\delta}} (h^2 - 2 h P + 2 \langle \nu, \nabla h \rangle)
\mathrm{d} A \mathrm{d} t.
\end{align}
  Combining with \eqref{error estimate} and
  \begin{equation}
    \int_{u^{- 1} (\mathcal{B})} \Delta \phi_{\delta} = \int_U \Delta
    \phi_{\delta} - \int_{u^{- 1} (\mathcal{A})} \Delta \phi_{\delta} =
    \int_{\partial U} \partial_{\eta} \phi_{\delta} - \int_{u^{- 1}
    (\mathcal{A})} \Delta \phi_{\delta},
  \end{equation}
  we obtain
\begin{align}
& \int_{\partial U} \partial_{\eta} \phi_{\delta} + C_0  \int_{t \in
\mathcal{A}} | \Sigma_t | \mathrm{d} t \\
\geqslant & \frac{1}{2}  \int_{t \in \mathcal{B}} \int_{\Sigma_t} \frac{|
\nabla u|}{\phi_{\delta}}  \left[ \frac{| \bar{\nabla}^2 u|^2}{| \nabla
u|^2} + 2 \mu - R_{\Sigma_t} - \frac{2}{| \nabla u|} (\langle p, \nabla^2
u \rangle + \langle \nabla u, \nabla P \rangle) \right] \mathrm{d} A
\mathrm{d} t \\
& \quad + \frac{1}{2}  \int_{t \in \mathcal{B}} \int_{\Sigma_t}
\frac{1}{\phi_{\delta}} (h^2 - 2 h P + 2 \langle \nu, \nabla h \rangle)
\mathrm{d} A \mathrm{d} t. \label{final integrated}
\end{align}
  \
  
  On the set $u^{- 1} (\mathcal{B})$, we have that $| \nabla u|$ is uniformly
  bounded from below. In addition, on $\partial_{\neq 0} U$ it holds that
  \begin{equation}
    \partial_{\eta} \phi_{\delta} = \frac{| \nabla u|}{\phi_{\delta}}
    \partial_{\eta} | \nabla u| \rightarrow \partial_{\eta} | \nabla u| 
    \text{ as } \delta \rightarrow 0.
  \end{equation}
  Therefore, the limit $\delta \to 0$ may be taken in \eqref{final
  integrated}, resulting in the same bulk expression except that
  $\phi_{\varepsilon}$ is replaced by $| \nabla u|$, and with the boundary
  integral taken over the restricted set. Furthermore, by Sard's theorem (see
  {\cite[Remark 3.3]{hirsch-spacetime-2020}}) the measure $|\mathcal{A}|$ of
  $\mathcal{A}$ may be taken to be arbitrarily small. Since the map $t \mapsto
  | \Sigma_t |$ is integrable over $[\underline{u}, \bar{u}]$ in light of the
  coarea formula, by then taking $|\mathcal{A}| \to 0$ we obtain
  
  Lastly integration by parts gives
  \begin{equation}
    \int_U \langle p, \nabla^2 u \rangle = \int_U p^{ij} \nabla_{ij} u =
    \int_{\partial U} p (\nabla u, \eta) - \int_U u^i \nabla^j p_{ij},
  \end{equation}
  and recalling that $J = \mathrm{div}_g (p - Pg)$ and $R_{\Sigma_t} = 2
  K_{\Sigma_t}$, yields the desired result.
\end{proof}

Now we discuss two special boundary conditions of the equation \eqref{h
spacetime harmonic} namely Dirichlet and Neumann boundary conditions and its
geometric implications on the boundary contribution
\begin{equation}
  \int_{\partial_{\neq 0} U} (\partial_{\eta} | \nabla u| + p (\nabla u,
  \upsilon)) \mathrm{d} A \label{boundary contribution}
\end{equation}
in \eqref{level set method main formula}. We assume that $u \in C^2 (U \cup
S)$ where $S$ is a relatively open portion of $\partial U$. This is a valid
assumption due to an existence theorem of {\cite[Section
4]{hirsch-spacetime-2020}}.

We have the following:

\begin{lemma}
  Assume the solution $u$ to \eqref{h spacetime harmonic} takes constant
  values on $S$, then
  \begin{equation}
    \partial_{\eta} | \nabla u| + p (\nabla u, \eta) = (- H_S
    -\ensuremath{\operatorname{tr}}_S p + h) | \nabla u| .
  \end{equation}
\end{lemma}

\begin{proof}
  Since $u$ is constant on $S$, then by using the decomposition of the
  Laplacian $\Delta$,
  \begin{equation}
    - P | \nabla u| + h | \nabla u| = \Delta u = H_S \langle \nabla u, \eta
    \rangle + \nabla^2 u (\eta, \eta) . \label{laplacian decomposition}
  \end{equation}
  The decomposition is already used in \eqref{laplacian decomposition on level
  set}. Since $u$ is constant on $S$, $\nabla u$ either point outward or
  inward of $U$. We calculate only the case when $\nabla u$ points outward,
  that is $\eta = \nabla u / | \nabla u|$,
\begin{align}
& \partial_{\eta} | \nabla u| + p (\nabla u, \eta) \\
= & \tfrac{1}{| \nabla u|} (\nabla^2 u) (\nabla u, \eta) + p (\nabla u,
\eta) \\
= & (\nabla^2 u) (\eta, \eta) + p (\eta, \eta) | \nabla u| \\
= & - H_S | \nabla u| - P | \nabla u| + p (\eta, \eta) | \nabla u| + h |
\nabla u| \\
= & (- H -\ensuremath{\operatorname{tr}}_S p + h) | \nabla u| .
\end{align}
  \ 
\end{proof}

\begin{lemma}
  \label{neumann prepare}Suppose that $\tfrac{\partial u}{\partial \eta} = 0$
  on $S$, then
  \begin{equation}
    \partial_{\eta} | \nabla u| = \tfrac{1}{| \nabla u|} B (\nabla u, \nabla
    u),
  \end{equation}
  where $B$ is the second fundamental form of $S$ in $U$.
\end{lemma}

\begin{proof}
  First,
\begin{align}
& \partial_{\eta} | \nabla u| \\
= & \tfrac{1}{| \nabla u|} \eta_j \nabla_i u \nabla^i \nabla^j u
\\
= & \tfrac{1}{| \nabla u|} [\nabla_i u \nabla^i (\eta_j \nabla^j u) -
\nabla^i u \nabla^j u \nabla^i \eta_j] \\
= & - \tfrac{1}{| \nabla u|} B (\nabla u, \nabla u),
\end{align}
  we have used the boundary condition $\partial u / \partial \eta = 0$.
\end{proof}

Note that we have not used $u$ is a solution to \eqref{h spacetime harmonic}.
On a regular level set $\Sigma_t$, let $e_1$ be a unit tangent vector of
$\partial \Sigma_t$, then $\{\eta, \nu = \tfrac{\nabla u}{| \nabla u|}, e_1
\}$ forms an orthonormal basis. We see that the geodesic curvature of
$\partial \Sigma_t$ in $\Sigma_t$ is given by
\begin{equation}
  \kappa_{\partial \Sigma_t} = \langle \nabla_{e_1} \eta, e_1 \rangle
\end{equation}
since $\tfrac{\partial u}{\partial \eta} = 0$ implies that $\eta$ is
orthogonal to $\partial \Sigma_t$ in $\Sigma_t$ and points outward of
$\Sigma_t$. So
\begin{align}
& B (\nu, \nu) = \langle \nabla_{\nu} \eta, \nu \rangle \\
= & (\langle \nabla_{\nu} \eta, \nu \rangle + \langle \nabla_{e_1} \eta, e_1
\rangle) - \langle \nabla_{e_1} \eta, e_1 \rangle \\
= & H_S - \kappa_{\partial \Sigma_t} .
\end{align}
Therefore, we conclude the following.

\begin{lemma}
  \label{neumann}Assume the solution $u$ to \eqref{h spacetime harmonic} takes
  constant values on $S$, at a point in $S \cap \partial \Sigma_t$ with
  $\Sigma_t$ being a regular level set, then
  \begin{equation}
    \partial_{\eta} | \nabla u| + p (\nabla u, \eta) = (- H_S + p
    (\tfrac{\nabla u}{| \nabla u|} {,} \eta) + \kappa_{\partial \Sigma_t}) |
    \nabla u| .
  \end{equation}
\end{lemma}

So to study rigidity questions on initial data sets with boundary, it is
natural to assume
\begin{equation}
  \mu + J (\nu) + h^2 - 2 h P + 2 | \nabla h| \geqslant 0,
\end{equation}
and the convexity condition on the boundary
\begin{equation}
  H_S \geqslant |\ensuremath{\operatorname{tr}}_S p - h|,
\end{equation}
or
\begin{equation}
  H_S \geqslant p (\eta, e_0)
\end{equation}
where $e_0$ is any unit vector on $S$. With $h = 0$, the conditions on the
boundary are termed boundary dominant energy conditions by
{\cite{almaraz-spacetime-2019}} in their study of initial data sets with a
noncompact boundary.

\

\text{{\bfseries{Acknowledgment}}} Research of Xiaoxiang Chai is supported by
KIAS Grants under the research code MG074402. I would also like to thank Tin
Yau Tsang (UCI) and Sven Hirsch (Duke) for discussions on the level set
technique of harmonic functions.

\end{document}